\documentclass[11pt,leqno]{amsart}
\usepackage{epsfig}
\usepackage{amssymb}
\usepackage{amscd}
\usepackage[matrix,arrow]{xy}
\usepackage{graphicx}

\newtheorem{theo}{Theorem}[section]
\newtheorem{lemma}[theo]{Lemma}
\newtheorem{defi}[theo]{Definition}
\newtheorem{prop}[theo]{Proposition}
\newtheorem{conj}[theo]{Conjecture}

\newtheorem{remark}[theo]{Remark}

\numberwithin{equation}{section}

\def\Z{\mathbb{Z}}

\def\pre-tr{\operatorname{pre-tr}}

\def\Hom{\operatorname{Hom}}
\def\End{\operatorname{End}}

\newcommand{\cQ}{{\mathcal Q}}

\newcommand{\cP}{{\mathcal P}}

\newcommand{\cM}{{\mathcal M}}

\newcommand{\cD}{{\mathcal D}}

\newcommand{\cA}{{\mathcal A}}
\newcommand{\cB}{{\mathcal B}}

\newcommand{\cC}{{\mathcal C}}
\newcommand{\cE}{{\mathcal E}}

\newcommand{\cX}{{\mathcal X}}

\newcommand{\un}{\underline}

\newcommand{\id}{\operatorname{id}}

\newcommand{\Set}{\operatorname{Set}}

\usepackage{epsf}
\usepackage{amscd}

\title[A proof of Kontsevich-Soibelman conjecture]
{A proof of Kontsevich-Soibelman conjecture}

\author{Alexander I. Efimov}
\address{Department of Mechanics and Mathematics, Moscow State University, Moscow,
Russia}
\address{Independent University of Moscow, Moscow,
Russia} \email{efimov@mccme.ru}
\thanks{The author was partially supported by
the Moebius Contest Foundation for Young Scientists, and by the
NSh grant 1983.2009.1.}

\begin{document}

\begin{abstract} It is well known that "Fukaya category" is in fact an $A_{\infty}\text{-}$pre-category in sense of Kontsevich
and Soibelman \cite{KS}. The reason is that in general the
morphism spaces are defined only for transversal pairs of
Lagrangians, and higher products are defined only for transversal
sequences of Lagrangians. In \cite{KS} it is conjectured that for
any graded commutative ring $k,$ quasi-equivalence classes of
$A_{\infty}\text{-}$pre-categories over $k$ are in bijection with
quasi-equivalence classes of $A_{\infty}\text{-}$categories over
$k$ with strict (or weak) identity morphisms.

In this paper we prove this conjecture for essentially small
$A_{\infty}\text{-}$(pre-)categories, in the case when $k$ is a
field. In particular, it follows that we can replace Fukaya
$A_{\infty}\text{-}$pre-category with a quasi-equivalent actual
$A_{\infty}\text{-}$category.

We also present natural construction of pre-triangulated envelope
in the framework of $A_{\infty}\text{-}$pre-categories. We prove
its invariance under quasi-equivalences.
\end{abstract}

\maketitle

\tableofcontents

\section{Introduction}

A remarkable construction of K. Fukaya \cite{F} associates to a
symplectic manifold a ($\Z\text{-}$ or $\Z/2\text{-}$)graded
$A_{\infty}\text{-}$pre-category in sense of Kontsevich and
Soibelman \cite{KS}. Its objects are Lagrangian submanifolds with
some additional structures. This is not an actual
$A_{\infty}\text{-}$category since in general the morphism spaces
are defined only for transversal pairs of Lagrangians, and higher
products are defined only for transversal sequences of
Lagrangians. Fukaya's construction is used in the categorical
interpretation of mirror symmetry \cite{K} for Calabi-Yau
varieties, and further generalizations to Fano and general cases,
the so-called homological mirror symmetry conjecture. For the
systematic exposition of different versions of Fukaya
$A_{\infty}\text{-}$pre-categories, see \cite{Se}.

However, in order to prove HMS conjecture at least in some special cases,
one should first replace Fukaya $A_{\infty}\text{-}$pre-category
with a (quasi-equivalent) actual $A_{\infty}\text{-}$category.
Clearly, each $A_{\infty}\text{-}$category (with weak identity
morphisms) can be considered also as an
$A_{\infty}\text{-}$pre-category. Kontsevich and Soibelman
\cite{KS} formulated the following natural conjecture.

\begin{conj}(\cite{KS})Let $k$ be a graded commutative ring.
Then quasi-equivalence classes of
$A_{\infty}\text{-}$pre-categories over $k$ are in bijection with
quasi-equivalence classes of $A_{\infty}\text{-}$categories over
$k$ with strict (or weak) identity morphisms.\end{conj}

The main result of this paper is the following theorem.

\begin{theo}\label{dublicate}Let $k$ be a field. Then quasi-equivalence classes of
essentially small $A_{\infty}\text{-}$pre-categories over $k$ are
in bijection with quasi-equivalence classes of essentially small
$A_{\infty}\text{-}$categories over $k$ with strict (or weak)
identity morphisms.
\end{theo}

We deal with $A_{\infty}\text{-}$(pre-)categories over a field
because we need to pass to minimal
$A_{\infty}\text{-}$(pre-)categories (i.e. with $m_1=0$). Further,
we deal with essentially small
$A_{\infty}\text{-}$(pre-)categories for purely set-theoretical
reason: we need to consider Hochshild cohomology of graded
(pre-)categories.

The second subject of the paper is the natural construction of
twisted complexes in the framework of
$A_{\infty}\text{-}$pre-categories. Here the main statement is the
invariance of twisted complexes under quasi-equivalences
(Proposition \ref{invariance}). For ordinary
$A_{\infty}\text{-}$categories we obtain standard pre-triangulated
envelopes.

The paper is organized as follows.

In Section \ref{Preliminaries} we define
$A_{\infty}\text{-}$categories, strict and weak identity
morphisms, quasi-equivalences, and
$A_{\infty}\text{-}$pre-categories, following \cite{KS}.

In Section \ref{main_theorem} we prove Main Theorem
\ref{dublicate} (Theorem \ref{main_theo}). The proof goes as
follows. In Subsection \ref{small} we pass from essentially small
$A_{\infty}\text{-}$(pre-)categories to small ones. Further, in
Subsection \ref{minimal} we pass from small to small minimal
$A_{\infty}\text{-}$(pre-)categories. In Subsection \ref{Hoch} we
define Hochshild cohomology of graded pre-categories. Roughly
speaking, obstructions to constructing, step by step, of
$A_{\infty}\text{-}$structures and $A_{\infty}\text{-}$morphisms
live in these cohomology spaces.

In Subsection \ref{main_lemma} we formulate and prove Main Lemma
(Lemma \ref{main_lemma}) about invariance of Hochshild cohomology
under equivalences of graded pre-categories. In contrast to
ordinary DG and $A_{\infty}\text{-}$categories, this is
non-trivial, and this is in fact the crucial point in the proof of
Main Theorem. Here we use the language of simplicial local
systems.

In Subsection \ref{A-structures} we introduce the sets of
equivalence classes of minimal $A_{\infty}\text{-}$structures on
graded pre-categories, and develop basic obstruction theory for
lifting $A_{\infty}\text{-}$structures and
$A_{\infty}\text{-}$homotopies. In Subsection \ref{inv_Theo} we
apply Main Lemma to prove the invariance of the set of equivalence
classes of minimal $A_{\infty}\text{-}$structures on graded
pre-categories. Finally, in Subsection \ref{proof} we prove Main
Theorem using the invariance result.

In Section \ref{twisted} we present the construction of
pre-triangulated envelope for $A_{\infty}\text{-}$pre-categories
over arbitrary graded commutative ring. We verify that it is
well-defined and is invariant under quasi-equivalences.

{\noindent{\bf Acknowledgments.}} I am grateful to D. Kaledin for
his remarks.

\section{Preliminaries on $A_{\infty}$-(pre-)categories}
\label{Preliminaries}

Fix some basic field $k$ of arbitrary characteristic.

\subsection{Non-unital $A_{\infty}$-algebras and $A_{\infty}$-categories}

Let $A=\bigoplus\limits_{i\in\Z}A^i$ be a $\Z$(resp.
$\Z/2$)-graded vector space. Denote by $A[n]$ its shift by $n:$
$A[n]^i:=A^{i+n}.$

\begin{defi}A structure of a non-unital $A_{\infty}\text{-}$algebra on $A$ is given by degree $+1$ coderivation
$b:T_+(A[1])\to T_+(A[1]),$ such that $b^2=0.$ Here
$T_+(A[1])=\bigoplus\limits_{n\geq 1}A[1]^{\otimes n}$ is a cofree
tensor coalgebra.\end{defi}

The coderivation $b$ is uniquely determined by its "Taylor
coefficients" $m_n:A^{\otimes n}\to A[2-n],$ $n\geq 1.$ The
condition $b^2=0$ is equivalent to the series of quadratic
relations

\begin{equation}\label{eq_for_m_n}\sum\limits_{i+j=n+2}\sum\limits_{0\leq l\leq
i-1}(-1)^{\epsilon}m_i(a_0,\dots,a_{l-1},m_j(a_l,\dots,a_{l+j-1}),a_{l+j},\dots,a_n)=0,
\end{equation}
where $a_m\in A,$ and $\epsilon=j\sum\limits_{0\leq s\leq
l-1}\deg(a_s)+l(j-1)+j(i-1).$ In particular, for $n=0,$ we have
$m_1^2=0.$

\begin{defi} An $A_{\infty}\text{-}$ morphism of non-unital $A_{\infty}\text{-}$algebras $A\to B$ is
a morphism of the corresponding non-counital DG coalgebras
$T_+(A[1])\to T_+(B[1]).$\end{defi}

Such an $A_{\infty}\text{-}$morphism is uniquely determined by its
"Taylor coefficients" $f_n:A^{\otimes n}\to B[1-n],$ satisfying
the following system of equations:

\begin{multline}\sum\limits_{1\leq l_1<l_2<\dots<l_i=n}(-1)^{\gamma_i}m_i^B(f_{l_1}(a_1,\dots,a_{l_1}),
\dots,f_{n-l_{i-1}}(a_{l_{i-1}+1},\dots,a_n))=\\
\sum\limits_{s+r=n+1}\sum\limits_{1\leq j\leq
s}(-1)^{\epsilon_j}f_s(a_1,\dots,a_{j-1},m_r^A(a_j,\dots,a_{j+r-1}),a_{j+r},\dots,a_n),
\end{multline}
where $a_m\in A,$ $\epsilon_j=r\sum\limits_{1\leq p\leq
j-1}\deg(a_p)+j-1+r(s-j),$ $\gamma_i=\sum\limits_{1\leq p\leq
i-1}(i-p)(l_p-l_{p-1}-1)+\sum\limits_{1\leq p\leq
i-1}\nu(l_p)\sum\limits_{l_{p-1}+1\leq q\leq l_p}\deg(a_q),$ where
we use the notation $\nu(l_p)=\sum\limits_{p+1\leq m\leq
i}(l_m-l_{m-1}-1),$ and put $l_0=0.$

\begin{defi}A non-unital $\Z$(resp. $\Z/2$)-graded $A_{\infty}\text{-}$category $\cC$ is given by
the following data:

1) A class of objects $Ob(\cC).$

2) For any two objects $X_1$ and $X_2,$ a $\Z$(resp.
$\Z/2$)-graded vector space of morphisms $\Hom(X_1,X_2).$

3) For any sequence of objects $X_0,\dots,X_n,$ a map of graded
vector spaces $m_n:\bigotimes\limits_{0\leq i\leq
n-1}\Hom(X_i,X_{i+1})\to \Hom(X_0,X_n)[2-n].$

This data us required to satisfy the following property: for each
finite collection of objects $X_0,\dots,X_N,$ $N\geq 0,$ the
graded vector space $\bigoplus\limits_{0\leq i,j\leq
N}\Hom(X_i,X_j)$ equipped with operations $m_n,$ becomes an
$A_{\infty}\text{-}$algebra.
\end{defi}

\begin{remark}A non-unital $A_{\infty}\text{-}$algebra can be considered as a non-unital $A_{\infty}\text{-}$category
with one object $X$ such that $\Hom(X,X)=A.$\end{remark}

\begin{remark}If $\cC$ is a non-unital $A_{\infty}\text{-}$category, then we have a "non-unital" graded category $H(\cC),$
which is defined by replacing the spaces of morphisms by their
cohomologies with respect to $m_1.$ Here "non-unital" means that
we may not have identity morphisms
$id_X\in\Hom^0_{H(\cC)}(X,X).$\end{remark}

\begin{defi}An $A_{\infty}\text{-}$functor between non-unital $A_{\infty}\text{-}$categories $F:\cC_1\to \cC_2$ is given by
the following data:

1) A map of classes of objects $F:Ob(\cC_1)\to Ob(\cC_2).$

2) For any finite sequence of objects $X_0,\dots,X_n$ in $\cC_1,$
a morphism of graded vector spaces $f_n:\bigotimes\limits_{0\leq
i\leq n-1}\Hom(X_i,X_{i+1})\to \Hom(F(X_0),F(X_n))[1-n].$

It is required that for any finite collection of objects
$X_0,\dots,X_N,$ $N\geq 0,$ the sequence $f_n,$ $n\geq 1,$ defines
an $A_{\infty}\text{-}$morphism

\begin{equation}\bigoplus\limits_{0\leq i,j\leq
N}\Hom(X_i,X_j)\to \bigoplus\limits_{0\leq i,j\leq
N}\Hom(F(X_i),F(X_j)).\end{equation}
\end{defi}

\subsection{Identity morphisms}

Now we define strict and weak identity morphisms.

\begin{defi}Let $\cC$ be a non-unital $A_{\infty}\text{-}$category, and $X\in Ob(\cC).$ A morphism $e\in \Hom^0(X,X)$
is called

a) a strict identity if $m_2(f,e)=f,$ $m_2(e,g)=g,$ for $n\ne 2$
$m_n(f_1,\dots,f_{i-1},e,f_{i+1},\dots,f_n)=0$ for any morphisms
$f, g, f_j$ such that the equalities make sense. In this case we
put $1_X:=e.$

b) a weak identity if $m_1(e)=0,$ and for any closed morphisms
$f:X\to Y,$ $g:Z\to X,$ we have that in $H(\cC)$
$\bar{e}\cdot\bar{g}=\bar{g},$ $\bar{f}\cdot \bar{e}=\bar{f}.$
\end{defi}

Clearly, a strict identity is also a weak identity.

\begin{remark}If a non-unital $A_{\infty}\text{-}$category $\cC$ has at least weak identity morphisms, then $H(\cC)$ is
an actual graded category, and vice versa.\end{remark}

\begin{defi}An $A_{\infty}\text{-}$functor $F:\cC\to\cD$ between $A_{\infty}\text{-}$categories with strict (resp. weak)
identity morphisms is called strictly (resp. weakly) unital if it
preserves strict units and $f_n(g_1,\dots,1_X,\dots,g_n)=0$
whenever $n>1$ (resp. if it preserves weak identity
morphisms).\end{defi}

\begin{defi}A strictly (resp. weakly) unital $A_{\infty}\text{-}$functor $F:\cC\to \cD$ between strictly (resp. weakly)
unital $A_{\infty}\text{-}$categories is called a
quasi-equivalence if the induced functor $H(F):H(\cC)\to H(\cD)$
is an equivalence of graded categories.

Two $A_{\infty}$-categories with weak identity morphisms $\cC$ and
$\cD$ are called quasi-equivalent if there exists a finite
sequence of $A_{\infty}\text{-}$categories with weak identity
morphisms $\cC_0=\cC,\cC_1,\dots,\cC_n=\cD$ such that for $0\leq
i\leq n-1$ there exists a quasi-equivalence $\cC_i\to\cC_{i+1}$ or
vice versa.
\end{defi}

The following statement is well-known, see \cite{L-H}.

\begin{prop}If we consider only $A_{\infty}\text{-}$categories with strict identity morphisms and strictly unital quasi-equivalences,
then the resulting quasi-equivalence classes are in bijection with
quasi-equivalence classes of $A_{\infty}\text{-}$categories with
weak identity morphisms.
\end{prop}

\subsection{$A_{\infty}$-pre-categories}

Now we recall the definition of $A_{\infty}\text{-}$pre-categories
which were originally defined in \cite{KS}. We start with the
notion of a non-unital $A_{\infty}\text{-}$pre-category.

\begin{defi}A non-unital $\Z$(resp. $\Z/2$)-graded $A_{\infty}\text{-}$pre-category $\cC$ is the following data:

a) A class of objects $Ob(\cC),$

b) For any $n\geq 1$ a subclass $\cC_{tr}^n\subset Ob(\cC)^n$ of
transversal sequences. It is required that $\cC_{tr}^1=Ob(\cC).$

c) For each pair $(X_1,X_2)\in \cC_{tr}^2,$ a graded vector space
$\Hom(X_1,X_2).$

d) For a transversal sequence of objects $X_0,\dots,X_n,$ a map of
graded vector spaces $m_n:\bigotimes\limits_{0\leq i\leq
n-1}\Hom(X_i,X_{i+1})\to \Hom(X_0,X_n)[2-n].$

It is required that each subsequence
$(X_{i_1},X_{i_2},\dots,X_{i_l}),$ $0\leq i_1<i_2<\dots<i_l\leq n$
of a transversal sequence $(X_0,\dots,X_N)$ is transversal, and
that the graded vector space $\bigoplus\limits_{0\leq i<j\leq
N}\Hom(X_i,X_j)$ with operations $m_n$ becomes a non-unital
$A_{\infty}\text{-}$algebra.\end{defi}

Note that the property of transversality for a pair of objects is
not required to be symmetric. Clearly, a non-unital
$A_{\infty}\text{-}$category is the same as a non-unital
$A_{\infty}\text{-}$pre-category $\cC$ with
$\cC_{tr}^n=Ob(\cC)^n.$

\begin{defi}An $A_{\infty}\text{-}$functor between non-unital $A_{\infty}\text{-}$pre-categories $F:\cC\to \cD$ is
given by
the following data:

1) A map of classes of objects $F:Ob(\cC)\to Ob(\cD),$ such that
$F(\cC_{tr}^n)\subset\cD_{tr}^n.$

2) For any finite transversal sequence of objects $X_0,\dots,X_n$
in $\cC,$ a morphism of graded vector spaces
$f_n:\bigotimes\limits_{0\leq i\leq n-1}\Hom(X_i,X_{i+1})\to
\Hom(F(X_0),F(X_n))[1-n].$

It is required that for any transversal sequence of objects
$X_0,\dots,X_N,$ $N\geq 0,$ the sequence $f_n,$ $n\geq 1,$ defines
an $A_{\infty}\text{-}$morphism

\begin{equation}\bigoplus\limits_{0\leq i<j\leq
N}\Hom(X_i,X_j)\to \bigoplus\limits_{0\leq i<j\leq
N}\Hom(F(X_i),F(X_j)).\end{equation}
\end{defi}

Roughly speaking, an $A_{\infty}\text{-}$pre-category is a
non-unital $A_{\infty}\text{-}$pre-category with sufficiently many
transversal sequences. To define this notion rigorously, we first
define the notion of a quasi-isomorphism.

\begin{defi}Let $\cC$ be a non-unital $A_{\infty}\text{-}$pre-category, and $(X_1,X_2)\in\cC_{tr}^2.$ A closed morphism
$f\in\Hom^0(X_1,X_2)$ is called a quasi-isomorphism if for any
objects $X_0$ and $X_3$ such that $(X_0,X_1,X_2)\in\cC_{tr}^3,$
$(X_1,X_2,X_3)\in\cC_{tr}^3,$ one has that the maps
\begin{equation}m_2(f,\cdot):\Hom(X_0,X_1)\to\Hom(X_0,X_2)\quad\text{and}\quad
m_2(\cdot,f):\Hom(X_2,X_3)\to\Hom(X_1,X_3)
\end{equation}
are quasi-isomorphisms.\end{defi}

\begin{defi} An $A_{\infty}\text{-}$pre-category is a non-unital
$A_{\infty}\text{-}$pre-category $\cC$ which satisfies the
following extension property:

For any finite collection of transversal sequences $(S_i)_{i\in
I}$ in $\cC$ and an object $X$ there exist objects $X_-,X_+$ and
quasi-isomorphisms $f_-:X_-\to X,$ $f_+:X\to X_+$ such that the
sequences $(X_-,S_i,X_+),$ $i\in I,$ are transversal.
\end{defi}

\begin{defi}Let $\cC$ and $\cD$ be $A_{\infty}\text{-}$pre-categories. An $A_{\infty}\text{-}$functor $F:\cC\to\cD$ is an
$A_{\infty}\text{-}$functor between the corresponding non-unital
$A_{\infty}\text{-}$pre-categories which takes quasi-isomorphisms
in $\cC$ to quasi-isomorphisms in $\cD.$\end{defi}

\begin{remark} An $A_{\infty}$-category with weak identity morphisms is the same as an $A_{\infty}\text{-}$pre-category
$\cC$ with $\cC_{tr}^n=Ob(\cC)^n.$\end{remark}

Now we define the notion of quasi-equivalence for
$A_{\infty}\text{-}$pre-categories.

\begin{defi} An $A_{\infty}\text{-}$functor $F:\cC\to \cD$ between
$A_{\infty}\text{-}$pre-categories is called quasi-equivalence if:

a) For any pair $(X_1,X_2)\in \cC_{tr}^2,$ the map
$f_1:\Hom(X_1,X_2)\to \Hom(F(X_1),F(X_2))$ is a quasi-isomorphism.

b) Each object $Y$ of $\cD$ is quasi-isomorphic to some $F(X),$
$X\in Ob(\cC).$
\end{defi}

\begin{defi}Two $A_{\infty}\text{-}$pre-categories $\cC$ and $\cD$ are called quasi-equivalent if there exists
a sequence $\cC_0,\cC_1,\dots,\cC_n$ of
$A_{\infty}\text{-}$pre-categories with $\cC_0=\cC,$ $\cC_n=\cD,$
such that for each $o\leq i\leq n-1$ there exists a
quasi-equivalence from $\cC_i$ to $\cC_{i+1}$ or vice versa.
\end{defi}

\section{Main Theorem}
\label{main_theorem}

\begin{defi}An $A_{\infty}\text{-}$(pre-)category $\cC$ is called essentially small (resp. small) if the
quasi-isomorphism classes in $\cC$ form a set (resp. $Ob(\cC)$ is
a set).
\end{defi}

In the rest of the paper $"A_{\infty}\text{-}$category" stands for
$"A_{\infty}\text{-}$category with weak identity morphisms".

\begin{theo}\label{main_theo}Quasi-equivalence classes of essentially small $A_{\infty}\text{-}$pre-categories are in bijection
with quasi-equivalence classes of essentially small
$A_{\infty}\text{-}$categories.\end{theo}

The proof of this theorem will occupy this section.

\subsection{From essentially small to small}
\label{small}

Define the quasi-equivalence classes of small
$A_{\infty}\text{-}$(pre-)categories by requiring all the
intermediate $A_{\infty}\text{-}$(pre-)categories $\cC_i$ in the
chain to be small.

\begin{lemma}\label{ess.small-small}
Quasi-equivalence classes of essentially small
$A_{\infty}\text{-}$(pre-)categories are in bijection with
quasi-equivalence classes of small
$A_{\infty}\text{-}$(pre-)categories.
\end{lemma}
\begin{proof}We will just show how to construct a small $A_{\infty}\text{-}$(pre-)category starting from essentially
small one. All the rest checkings are straightforward.

The case of $A_{\infty}\text{-}$categories is obvious. Namely,
given essentially small $A_{\infty}\text{-}$category $\cC,$ we can
choose one object from each quasi-isomorphism class, and take the
corresponding full $A_{\infty}\text{-}$subcategory. By
construction, it is small and the inclusion
$A_{\infty}\text{-}$functor is a quasi-equivalence.

To treat the case of $A_{\infty}\text{-}$pre-categories, we need
the following non-canonical operation on the small subclasses. Let
$\cC$ be an $A_{\infty}\text{-}$pre-category, and $\cD\subset
Ob(\cC)$ be a small subclass (i.e. which is a set). For each
finite collection $S_1,\dots,S_n$ of transversal sequences in
$\cD,$ choose the objects $X_{\pm}=X_{\pm}(S_{\cdot})\in Ob(\cC)$
such that all the sequences $(X_-,S_i,X_+)$ are transversal. Let
$T(\cD)$ be the union of $\cD$ and all pairs
$(X_-(S_{\cdot}),X_+(S_{\cdot})).$ Then $T(\cD)$ is again a small
subclass of $Ob(\cC).$ Put
\begin{equation}T^{\infty}(\cD):=\bigcup\limits_{n=0}^{\infty}T^n(\cD).\end{equation}

Now, let $\cD$ be a small subclass of $Ob(\cC)$ which contains
precisely one object from each quasi-isomorphism class. Take the
full $A_{\infty}\text{-}$sub-pre-category
$\cE=T^{\infty}(\cD)\subset \cC.$ Clearly, $\cE$ is a small
$A_{\infty}\text{-}$pre-category by construction, and the
inclusion $A_{\infty}\text{-}$functor $\cE\hookrightarrow \cC$ is
a quasi-equivalence.
\end{proof}

In the rest of this section, we deal only with small
$A_{\infty}\text{-}$(pre-)categories.

\subsection{Minimal models}
\label{minimal}

\begin{defi}An $A_{\infty}$-(pre-)category is called minimal if $m_1=0.$\end{defi}

Define quasi-equivalence classes of minimal
$A_{\infty}\text{-}$(pre-)categories by requiring all the
intermediate $A_{\infty}\text{-}$(pre-)categories $\cC_i$ in the
chain to be minimal.

\begin{lemma}\label{red_to_min} Quasi-equivalence classes of $A_{\infty}\text{-}$(pre-)categories are in
bijection with quasi-equivalence classes of minimal
$A_{\infty}\text{-}$(pre-)categories.
\end{lemma}
\begin{proof}Let $\cC$ be an $A_{\infty}\text{-}$pre-category. For each pair of objects $(X_1,X_2)$
(resp. for $(X_1,X_2)\in \cC_{tr}^2$) choose a direct sum
decomposition $\Hom(X_1,X_2)=K(X_1,X_2)\oplus Ac(X_1,X_2),$ where
$K(X_1,X_2)$ is a subcomplex with zero differential, and
$Ac(X_1,X_2)$ is an acyclic subcomplex. Denote by
$i(X_1,X_2):K(X_1,X_2)\to \Hom(X_1,X_2)$ the natural inclusion,
$p(X_1,X_2):\Hom(X_1,X_2)\to Ac(X_1,X_2)$ the natural projection
(both $i$ and $p$ are quasi-isomorphisms). Choose some contracting
homotopy $h=h(X_1,X_2):Ac(X_1,X_2)\to Ac(X_1,X_2),$ such that
$h^2=0.$ Finally, denote by $H(X_1,X_2):\Hom(X_1,X_2)\to
\Hom(X_1,X_2)$ the extension of $h$ by zero.

Starting from the data $i(X_1,X_2),$ $p(X_1,X_2)$ and
$H(X_1,X_2),$ and applying the standard formulae for transferring
$A_{\infty}\text{-}$structures (see \cite{KS}), one obtains:

1) An $A_{\infty}$-(pre-)category $\cC_{min}$ with
$Ob(\cC_{min})=Ob(\cC)$ and
$\Hom_{\cC_{min}}(X_1,X_2)=K(X_1,X_2).$ Also, in the case of
$A_{\infty}\text{-}$pre-categories, we have
$(\cC_{min})_{tr}^n=\cC_{tr}^n.$

2) A quasi-equivalence $F:\cC_{min}\to \cC,$ such that $F(X)=X$
and $f_1=i.$

3)  A quasi-equivalence $G:\cC\to \cC_{min},$ such that $G(X)=X$
and $g_1=p.$

Lemma follows easily.
\end{proof}

Hence, it suffices to deal only with minimal
$A_{\infty}\text{-}$(pre-)categories.

\subsection{Hochshild cohomology of graded pre-categories}
\label{Hoch}

\begin{defi}Define a ($\Z$- or $\Z/2$-)graded pre-category as an $A_{\infty}\text{-}$pre-category with $m_n=0$ for
$n\ne 2.$

Define a functor between graded pre-categories to be an
$A_{\infty}\text{-}$functor with $f_n=0$ for $n\ne 1.$ We say that
such a functor is an equivalence if it is a quasi-equivalence of
$A_{\infty}\text{-}$pre-categories.\end{defi}

\begin{prop}Let $\cC$ be a graded pre-category. Then it can be canonically extended to an actual graded category
$\cC_{full}$ with $Ob(\cC_{full})=Ob(\cC),$ together with a
natural equivalence of graded pre-categories
$\iota:\cC\to\cC_{full},$ $\iota(X)=X$ for any $X\in
Ob(\cC).$\end{prop}
\begin{proof}Category $\cC_{full}$ can be obtained by formal inverting of quasi-isomorphisms in $\cC.$ We leave the
straightforward checking to the reader.\end{proof}

Let $\cC$ be a graded pre-category. Define the bigraded Hochshild
complex of $\cC$ by the following formula:

\begin{equation}CC^{i,j}(\cC)=\prod\limits_{(X_0,\dots,X_i)\in\cC_{tr}^i}
\Hom^j(\bigotimes\limits_{1\leq l\leq
i}\Hom_{\cC}(X_{l-1},X_l),\Hom_{\cC_{full}}(X_0,X_i)).\end{equation}

Here $i\in \Z_{\geq 0},$ and $j\in \Z$ (resp. $j\in \Z/2$). We
write $\cC_{full}$ in the above equation, because for $i=0$ the
sequence $(X_0,X_0)$ is not transversal in general.

The differential is the standard Hochshild one. It maps
$CC^{i,j}(\cC)$ to $CC^{i+1,j}(\cC).$ Namely, for $\phi\in
CC^{n,p}(\cC),$ we have
\begin{multline}
\partial(\phi)(a_{n+1},\dots,a_1)=\sum\limits_{i=1}^n
(-1)^{n-i}\phi(a_{n+1},\dots,a_{i+1}a_i,\dots,a_1)+\\(-1)^n\phi(a_{n+1},\dots,a_2)a_1+
(-1)^{p\deg(a_{n+1})+1}
a_{n+1}\phi(a_n,\dots,a_1).
\end{multline}

  Therefore, we have bigraded
Hochshild cohomology $HH^{i,j}(\cC).$

Now suppose that we have a fully faithful functor $F:\cC\to \cD$
between graded pre-categories. Clearly, it induces a morphism of
complexes (preserving the gradings)
\begin{equation}F^*:CC^{\cdot,\cdot}(\cD)\to CC^{\cdot,\cdot}(\cC).\end{equation}

\subsection{Main Lemma}
\label{main_lemma}

The key statement in our proof of Main Theorem is the following.

\begin{lemma}\label{iso_on_Hoch}Let $F:\cC\to\cD$ be an equivalence of graded
pre-categories. Then the induced morphism of complexes
\begin{equation}F^*:CC^{\cdot,\cdot}(\cD)\to CC^{\cdot,\cdot}(\cC)\end{equation}
induces isomorphisms $F^*:HH^{i,j}(\cD)\to HH^{i,j}(\cC).$
\end{lemma}

We will prove this lemma in the framework of local systems on
simplicial sets.  Recall the category $\Delta,$ which is a
(non-full) subcategory of the category of sets. Its objects are
denoted by $[n]=\{0,1\dots,n\},$ where $n\in \Z_{\geq 0}.$
Further,
\begin{equation}\Hom_{\Delta}([m],[n])=\{\text{non-decreasing maps of sets }f:[m]\to [n]\}.\end{equation}

The category $\Delta$ admits the well-known system of generating
morphisms, which consists of the face maps $\sigma_i^n:[n]\to
[n+1],$ $i=0,1,\dots,n+1,$ and degeneration maps $s_i^n:[n]\to
[n-1],$ $i=0,1\dots,n-1.$ Here
\begin{equation}\sigma_i^n(j)=\begin{cases}j & \text{for }j\leq i-1\\
j+1 & \text{for }j\geq i,\end{cases}\quad
s_i^n(j)=\begin{cases}j & \text{for }j\leq i\\
j-1 & \text{for }j\geq i+1.\end{cases}\end{equation}

A simplicial set $\cX_{\cdot}$ is a functor $\cX:\Delta^{op}\to
\Set,$ $\cX_n=\cX([n]).$ We treat $\cX$ also as a category (which
we denote by the same letter). Its objects are elements $X_n\in
\cX_n.$ Further,
\begin{equation}\Hom_{\cX}(X_m,X_n)=\{f:[m]\to [n] |
f^{*}(X_n)=X_m\},
\end{equation}
and the composition is obvious.

Given a simplicial set $\cX_{\cdot},$ a (cohomological) local
system of $k\text{-}$vector spaces on $\cX_{\cdot}$ is a functor
$\cA:\cX\to k\text{-}Vect.$ In particular, we have a constant
local system $\un{k}$ with fiber $k.$ Namely, $\un{k}(X_n)=k,$ and
$\un{k}(f)=\id$ for $f:X_m\to X_n.$

Local systems obviously form an abelian category. It is also easy
to see that it has enough projectives. Indeed, for each object
$X_n\in \cX,$ we have the corresponding local system $P_{X_n},$
$P_{X_n}(Y_m)=k[\Hom_{\cX}(X_n,Y_m)],$ where $k[\cdot]$ means the
formal linear envelope. By Ioneda Lemma, we have that
\begin{equation}\Hom(P_{X_n},\cA)=\cA(X_n).\end{equation}
Hence, $P_{X_n}$ is projective. Each local system has a resolution
by direct sums of such projectives.

Cohomology of a local system $\cA$ is defined as
$Ext^{\cdot}(\un{k},\cA).$

Now recall the well-known projective resolution of the local
system $\un{k}.$ Put
\begin{equation}\cP_n=\bigoplus\limits_{X_n\in\cX_n}P_{X_n}.\end{equation} The differential $\partial:\cP_n\to\cP_{n-1}$ is
the sum of maps $(-1)^i\sigma_i^{n-1}:P_{X_n}\to
P_{(\sigma_i^{n-1})^*(X_n)}.$ The standard complex computing
cohomology of $\cA$ is just a complex of morphisms from this
projective resolution to $\cA.$

\begin{proof}[Proof of lemma \ref{iso_on_Hoch}.] Define graded category $\cE$ to be a full subcategory of $\cD_{full},$
which contains exactly one object from each isomorphism class in
$\cD.$ Choose some functor $G:\cD_{full}\to\cE,$ which equals to
identity on $\cE,$ and denote by the same letter its restriction
to $\cD.$ To proof the desired quasi-isomorphism, it suffices to
proof that maps $G^*:CC^{\cdot,\cdot}(\cE)\to
CC^{\cdot,\cdot}(\cD)$ and $(GF)^*:CC^{\cdot,\cdot}(\cE)\to
CC^{\cdot,\cdot}(\cC)$ are quasi-isomorphisms.

Therefore, we may assume that $\cD$ is itself an actual graded
category, and the functor $F$ induces a bijection between the set
of quasi-isomorphism classes of objects in $\cC$ and the set
$Ob(\cD).$

Now define the simplcial set $\cX_{\cdot}$ by the formula
$\cX_n=Ob(\cD)^{n+1}.$ Further, for $(Y_0,\dots,Y_n)\in\cX_n$ and
$f:[m]\to [n],$ put
$f^*(Y_0,\dots,Y_n)=(Y_{f(0)},Y_{f(1)},\dots,Y_{f(m)}).$

Define the local system $\cA$ on $\cX$ by the formula
\begin{equation}\cA(Y_0,\dots,Y_n)=\bigoplus\limits_{j}\Hom^j(\bigotimes\limits_{1\leq l\leq
n}\Hom_{\cD}(Y_{l-1},Y_l),\Hom_{\cD}(Y_0,Y_n)).\end{equation}

Further, for $f:[m]\to [n],$ and homogeneous
$\phi^m\in\cA(f^*(Y_0,\dots,Y_n)),$ $a_i\in\Hom_{\cD}(Y_{i-1},Y_n)$ we
put
\begin{multline}f(\phi^m)(a_n,\dots,a_1)=\\(-1)^{\epsilon(f)} a_n\dots a_{f(m)+1}
\phi^m(a_{f(m)}\dots a_{f(m-1)+1},\dots,a_{f(1)}\dots
,a_{f(0)+1})a_{f0)}\dots a_1,\end{multline} where $\epsilon
(f)=\deg(\phi^m)\sum\limits_{i=f(m)+1}^{n}\deg(a_i)+\frac{(n-m)(n+m+1)}2,$
and the product over the empty set equals to the corresponding
identity morphism.

Take the standard projective resolution $\cP_{\cdot}$ of the local
system $\un{k},$ as above. Then the complex
$\Hom(\cP_{\cdot},\un{k})$ is naturally isomorphic to the complex
$\bigoplus\limits_j CC^{\cdot,j},$ hence
\begin{equation}H^{\cdot}(\cA)\cong \bigoplus\limits_j HH^{\cdot,j}(\cD).\end{equation}

Now we would like to express $HH^{\cdot,\cdot}(\cC)$ in similar
terms. For $S\in \cC_{tr}^{n+1},$ put $Q_S:=P_{F(S)}$ --- a
projective local system on $\cX.$ If $S=(Y_0,\dots,Y_n),$ then put
$(\sigma_i^{n-1})^*(S):=(Y_0,\dots,\widehat{Y_i},\dots,Y_n).$
Clearly, $F((\sigma_i^{n-1})^*(S))=(\sigma_i^{n-1})^*(F(S)).$

Define a chain complex $\cQ_{\cdot}$ as follows. Put
$\cQ_n=\bigoplus\limits_{S\in \cC_{tr}^{n+1}}Q_S.$ Further, the
differential $\partial:\cQ_n\to\cQ_{n-1}$ is the sum of maps
$(-1)^i\sigma_i^{n-1}:\cQ_{S}\to \cQ_{(\sigma_i^{n-1})^*(S)}.$ We
have obvious projection $\cQ_0/Im(\partial)\to \un{k}.$

It is clear that the complex $\Hom(\cQ_{\cdot},\cA)$ is naturally
quasi-isomorphic to the complex
$\bigoplus\limits_jCC^{\cdot,j}(\cC).$ We have an obvious morphism
of complexes $\Phi:\cQ_{\cdot}\to \cP_{\cdot},$ with non-zero
components being identity maps $Q_S\to P_{F(S)}.$ The morphism
$\Phi$ is compatible with projections $\cQ_0\to \un{k},$ $\cP_0\to
\un{k}.$ Also, we have the commutative diagram

\begin{equation}
\begin{CD}
\bigoplus\limits_jCC^{i,j}(\cD) @>F^*>> \bigoplus\limits_jCC^{i,j}(\cD)\\
@V\cong VV                              @V\cong VV\\
\Hom(\cP_i,\cA) @>\cdot\circ\Phi >> \Hom(\cQ_i,\cA).
\end{CD}
\end{equation}

Therefore, we are left to prove that $\cQ_{\cdot}$ is a resolution
of $\un{k}.$

\noindent{\bf Sublemma.}{\it The complex $\cQ_{\cdot}$ is a
resolution of $\un{k}.$}

\begin{proof}For convenience put $\cQ_{-1}=\un{k},$ and treat $k=\un{k}(Y_0,\dots,Y_m)$ as $k[\Hom([-1],[n])],$
where $[-1]:=\varnothing.$ So we need to prove that the complex
$\cQ_{\cdot}$ is acyclic.

 Fix some sequence $T=(Y_0,\dots,Y_n)\in
\cX_n,$ $n\geq 0.$ and consider the complex $\cQ_{\cdot}(T)$ of
vector spaces. We have that $\cQ_m(T)$ consists of linear
combinations of pairs $(f,S),$ where $f:[m]\to [n],$ $S\in
\cC_{tr}^{m+1},$ such that $F(S)=f^*(T).$ Take some closed element
$a=\sum\limits_{j=1}^p\lambda_j(f_j,S_j)\in \cQ_m(T),$ $m\geq -1.$
We would like to construct its (non-canonically defined) pre-image
$H(a)\in \cQ_{m+1}(T)$ under the differential $\partial.$

Define the maps $g_j:[m+1]\to [n],$ $j=1,\dots,p,$ by the formula
\begin{equation}g_j(l)=\begin{cases}f_j(l-1) & \text{for }l>0\\
0 & \text{for }l=0.\end{cases}
\end{equation}
Further, by the definition of an $A_{\infty}\text{-}$pre-category,
there exists an object $\widetilde{Y}\in Ob(\cC)$ and a
quasi-isomorphism $\widetilde{Y}\to Y_0$ such that all sequences
$(\widetilde{Y},S_j)$ are transversal. Since $F$ induces a
bijection between quasi-isomorphism classes in $\cC$ and objects
in $\cD,$ we have that $F(\widetilde{Y})=F(Y_0).$ Therefore,
$F(\widetilde{Y},S_j)=(F(Y_0),F(S_j))=g_j^*(T).$ Put
\begin{equation}H(a)=\sum\limits_{j=1}^p\lambda_j(g_j,(\widetilde{Y},S_j)).\end{equation}
Then from the closedness of $a$ we immediately obtain that
$\partial(H(a))=a.$ This proves Sublemma.
\end{proof}
Lemma is proved.
\end{proof}

\subsection{$A_{\infty}\text{-}$structures on a graded pre-category}
\label{A-structures}

\begin{defi}An $A_{\infty}\text{-}$structure on a graded pre-category $\cC$ is a collection of maps $m_n,$ $n\geq 3,$ $\deg(m_n)=2-n$ for all transversal
sequences, such that together with $m_2(a,b)=ab$ and $m_1=0$ they
give a structure of $A_{\infty}\text{-}$pre-category on $\cC.$

Two $A_{\infty}\text{-}$structures $m$ and $m'$ on $\cC$ are
called strongly homotopic if there exists an
$A_{\infty}\text{-}$morphism $F:(\cC,m)\to (\cC,m')$ with $F(X)=X$
for $X\in Ob(\cC),$ and $f_1=\id.$ In this case $F$ is called a
strong homotopy between $m$ and $m'.$
\end{defi}

Formal collections of maps $f_n:\bigotimes_{1\leq i\leq
n}\Hom(X_{i-1},X_i)\to \Hom(X_0,X_n)$ of degree $1-n$ for all
transversal sequences $(X_0,\dots,X_n)\in\cC_{tr}^{n+1},$ $n\geq
1,$ with $f_1=\id,$ form a group $G_{\cC}.$ The product of $f$ and
$g$ is given by the same formula as the composition of
$A_{\infty}\text{-}$functors. This group acts on the set
$A_{\infty}(\cC)$ of $A_{\infty}\text{-}$structures on $\cC.$
Namely, $f(m)=m'$ iff $f$ is an $A_{\infty}\text{-}$morphism from
$(\cC,m)$ to $(\cC,m').$ Tautologically, two
$A_{\infty}\text{-}$structures are strongly homotopic iff they lie
in the same orbit of this action.

The following Lemmas are well-known for
$A_{\infty}\text{-}$algebras, and the proof is in fact
straightforward. We omit the proof.

\begin{lemma}\label{obstr_m_n}Let $(m_3,\dots,m_{n-1})$ be partially defined $A_{\infty}\text{-}$structure on a graded pre-category $\cC,$
i.e. the maps $m_{\leq n-1}$ satisfy all the required equations
which do not contain $m_{\geq n}.$ Write the first
$A_{\infty}\text{-}$constraint containing $m_n$ in the form
\begin{equation}\partial(m_n)=\Phi,\end{equation}
where $\partial$ is the Hochshild differential and
$\Phi=\Phi(m_3,\dots,m_{n-1})$ is quadratic expression. Then we
always have $\partial(\Phi)=0.$\end{lemma}

\begin{lemma}\label{perturb_m_n}Let $m$ and $m'$ be two $A_{\infty}$-structures on a graded pre-category $\cC,$ with
$m_i=m_i'$ for $i<n.$ Let $f:(\cC,m)\to (\cC,m')$ be an
$A_{\infty}\text{-}$morphism with $f_1=\id,$ and $f_i=0$ for
$2\leq i\leq n-2.$ Then $m_i'=m_i$ for $i\leq n-1,$ and
$m_n'=m_n+\partial(f_{n-1}).$\end{lemma}

\begin{lemma}\label{obstr_f_n}Let $m$ and $m'$ be two $A_{\infty}$-structures on a graded pre-category $\cC.$ Suppose
that $(f_1=\id,f_2,\dots,f_{n-1})$ is a partially defined strong
homotopy between $m$ and $m'.$ i.e. the maps $f_{\leq n-1}$
satisfy all the required equations which do not contain $f_{\geq
n}.$ Write the first $A_{\infty}\text{-}$constraint containing
$f_n$ in the form
\begin{equation}\partial(f_n)=\Psi,\end{equation}
where $\partial$ is the Hochshild differential and
$\Psi=\Psi(f_2,\dots,f_{n-1};m,m')$ is a polynomial expression.
Then we always have $\partial(\Psi)=0.$
\end{lemma}

We will also need the notion of homotopy between two
$A_{\infty}\text{-}$functors. First, let $f,f':A\to B$ be two
$A_{\infty}\text{-}$morphisms of (possibly non-unital)
$A_{\infty}\text{-}$algebras. We have the associated morphisms of
DG coalgebras $f,f':T_{+}(A[1])\to T_+(B[1]).$ A homotopy between
$f$ and $f'$ is a map $H:T_+(A[1])\to T_+(B[1])$ satisfying the
identities
\begin{equation}\label{homotopy}\Delta\circ H=(f\otimes H+H\otimes f')\circ \Delta,\end{equation}
and \begin{equation}f-f'=b_B\circ H+H\circ b_A.\end{equation}

Any map $H$ satisfying \eqref{homotopy} is uniquely determined by
its components $h_n:A^{\otimes n}\to B,$ $\deg(h_n)=-n.$

\begin{defi}Let $F,F':\cC\to\cD$ be $A_{\infty}\text{-}$functors between $A_{\infty}\text{-}$pre-categories,
such that $F(X)=F'(X)$ for each $X\in Ob(\cC).$ A homotopy $H$
between $f$ and $f'$ is a collection of maps
$h_n:\bigotimes\limits_{1\leq i\leq n}\Hom_{\cC}(X_{i-1},X_i)\to
\Hom_{\cD}(F(X_0),F(X_n))$ of degree $-n,$ for all transversal sequences
$(X_0,\dots,X_n),$ satisfying the following property. For each
transversal sequence $(X_0,\dots,X_N)\in\cC_{tr}^{n+1},$ we have
that the maps $h_n$ define a homotopy between the restricted
$A_{\infty}\text{-}$functors
\begin{equation}F,F':\bigoplus\limits_{i<j}\Hom_{\cC}(X_i,X_j)\to F,F':\bigoplus\limits_{i<j}\Hom_{\cD}(F(X_i),F(X_j)).
\end{equation}\end{defi}

We will need the following Lemma.

\begin{lemma}\label{perturb_f_n}Let $F:\cC\to\cD$ be an $A_{\infty}\text{-}$functor. Suppose that
we are given with a collection of maps
$h_n:\bigotimes\limits_{1\leq i\leq n}\Hom_{\cC}(X_{i-1},X_i)\to
\Hom_{\cD}(F(X_0),F(X_n))$ for all transversal sequences. Then
there exists a unique $A_{\infty}\text{-}$functor $F':\cC\to\cD,$
$F'(X)=F(X),$ such that the sequence $h_n$ defines a homotopy
between $F$ and $F'.$

Moreover, in the case when $\cC=(\cE,m),$ $\cD=(\cE,m')$ are
$A_{\infty}\text{-}$structures on the same graded pre-category
$\cE,$ $F$ belongs to $G_{\cE},$ and $h_i=0$ for $1\leq i\leq
n-2,$ we have that
\begin{equation}\begin{cases}f_i'=f_i & \text{for }1\leq i\leq
n-1,\\
f_n'=f_n+\partial(h_{n-1})\end{cases}\end{equation}
\end{lemma}
\begin{proof} The first statement is a direct consequence of \cite{P}, Lemma 2.1.
The second one is checked straightforwardly.\end{proof}

\subsection{Invariance Theorem}
\label{inv_Theo}

 Let $\phi:\cC\to\cD$ be an equivalence of graded pre-categories.
We have a natural map $\phi^*:A_{\infty}(\cD)\to A_{\infty}(\cC),$
and a homomorphism $\phi^*:G_{\cD}\to G_{\cC},$ compatible with
our group actions. Therefore we have a map of strong homotopy
equivalence classes of $A_{\infty}\text{-}$structures:
\begin{equation}\label{transfer}\phi^*:A_{\infty}(\cD)/G_{\cD}\to A_{\infty}(\cC)/G_{\cC}.\end{equation}

\begin{theo}\label{invar_theo}The map \eqref{transfer} is a bijection.\end{theo}

\begin{proof}{\bf Surjectivity.} First we prove that our map is surjective. Take some $A_{\infty}\text{-}$structure
$m$ on $\cC.$ We want to prove that it is strongly homotopic to
some $A_{\infty}\text{-}$structure of the form
$\phi^*(\widetilde{m}),$ where $\widetilde{m}$ is an
$A_{\infty}\text{-}$structure on $\cD.$ Clearly, it suffices to
prove the following Lemma.
\begin{lemma} Let $m$ be an $A_{\infty}\text{-}$structure on $\cC$
such that $m_i=\phi^*(\widetilde{m_i})$ for $3\leq i\leq n-1,$
where $(\widetilde{m_3},\dots,\widetilde{m_{n-1}})$ is a partially
defined $A_{\infty}\text{-}$structure on $\cD.$ Then $m$ is
strongly homotopic to some $A_{\infty}\text{-}$structure $m'$ such
that $m_i'=m_i$ for $i\leq n-1,$ and
$m_n'=\phi^*(\widetilde{m_n}),$ so that
$(\widetilde{m_3},\dots,\widetilde{m_n})$ is a partially defined
$A_{\infty}\text{-}$structure on $\cD.$ Moreover, strong homotopy
$(f_1,f_2,\dots)$ between $m$ and $m'$ can be taken to be such
that $f_2=\dots=f_{n-2}=0.$
\end{lemma}
\begin{proof}Write the first
$A_{\infty}\text{-}$constraint containing $m_n$ in the form
\begin{equation}\partial(m_n)=\Phi,\end{equation}
as in Lemma \ref{obstr_m_n}. We have that
$\Phi(m_3,\dots,m_{n-1})$ is a Hochshild coboundary. By Lemma
\ref{obstr_m_n}, we have that
$\Phi(\widetilde{m_3},\dots,\widetilde{m_{n-1}})$ is a Hochchild
cocycle. Further, we have that
\begin{equation}\Phi(m_3,\dots,m_{n-1})=\phi^*\Phi(\widetilde{m_3},\dots,\widetilde{m_{n-1}}).\end{equation}
Therefore, Lemma \ref{iso_on_Hoch} implies that
$\Phi(\widetilde{m_3},\dots,\widetilde{m_{n-1}})$ is a Hochshild
coboundary. Take some $\widetilde{m_n}\in CC^{n,2-n}(\cD)$ such
that
$\partial(\widetilde{m_n})=\Phi(\widetilde{m_3},\dots,\widetilde{m_{n-1}}).$

We have that $\phi^*(\widetilde{m_n})-m_n$ is a Hochshild cocycle.
Again by Lemma \ref{iso_on_Hoch}, we can choose $\widetilde{m_n}$
in such a way that this difference is a Hochshild coboundary. Take
some element $f\in G_{\cC}$ such that $f_2=\dots=f_{n-2}=0,$ and
$\partial(f_{n-1})=\phi^*(\widetilde{m_n})-m_n.$ Put $m'=f(m).$ By
Lemma \ref{perturb_m_n}, we have that
$m_i'=\phi^*(\widetilde{m_i})$ for $i\leq n.$ This proves
Lemma.\end{proof} Surjectivity is proved.

{\noindent}{\bf Injectivity.} We are left to prove that our map is
injective. Let $m,m'$ be $A_{\infty}\text{-}$structures on $\cD,$
and let $F:(\cC,\phi^*(m))\to (\cC,\phi^*(m'))$ be a strong
homotopy. We need to prove the existence of a strong homotopy
between $m$ and $m'.$ Clearly, it suffices to prove the following
Lemma.
\begin{lemma}Let $m,m'$ be $A_{\infty}\text{-}$structures on $\cC.$ Let $f$ be a strong
homotopy between $\phi^*(m)$ and $\phi^*(m').$ Suppose that
$f_i=\phi^*(\widetilde{f_i})$ for $2\leq i\leq n-1,$ where
$(\widetilde{f_2},\dots,\widetilde{f_{n-1}})$ is a partially
defined strong homotopy between $m$ and $m'.$ Then there exists
some strong homotopy $f'$ between $\phi^*(m)$ and $\phi^*(m')$
such that $f_i'=f_i$ for $i\leq n-1,$ and
$f_n'=\phi^*(\widetilde{f_n}),$ so that
$(\widetilde{f_2},\dots,\widetilde{f_n})$ is a partially defined
strong homotopy between $m$ and $m'.$\end{lemma}

\begin{proof}Write the first
$A_{\infty}\text{-}$constraint containing $f_n$ in the form
\begin{equation}\partial(f_n)=\Psi,\end{equation}
as in Lemma \ref{obstr_f_n}. We have that
$\Psi(f_2,\dots,f_{n-1};\phi^*(m),\phi^*(m'))$ is a Hochshild
coboundary. By Lemma \ref{obstr_f_n}, we have that
$\Psi(\widetilde{f_2},\dots,\widetilde{f_{n-1}};m,m')$ is a
Hochchild cocycle. Further, we have that
\begin{equation}\Psi(f_2,\dots,f_{n-1},\phi^*(m),\phi^*(m'))=\phi^*\Psi(\widetilde{f_2},\dots,\widetilde{f_{n-1}};m,m').\end{equation}
Therefore, Lemma \ref{iso_on_Hoch} implies that
$\Psi(\widetilde{f_2},\dots,\widetilde{f_{n-1}},m,m')$ is a
Hochshild coboundary. Take some $\widetilde{f_n}\in
CC^{n,1-n}(\cD)$ such that
$\partial(\widetilde{f_n})=\Psi(\widetilde{f_2},\dots,\widetilde{f_{n-1}};m,m').$

We have that $\phi^*(\widetilde{f_n})-f_n$ is a Hochshild cocycle.
Again by Lemma \ref{iso_on_Hoch}, we can choose $\widetilde{f_n}$
in such a way that this difference is a Hochshild coboundary. Take
some sequence of elements $h_n\in CC^{n,-n},$ $n\geq 1,$ such that
$h_2=\dots=h_{n-2}=0,$ and
$\partial(h_{n-1})=\phi^*(\widetilde{f_n})-f_n.$ By Lemma
\ref{perturb_f_n}, there exists a unique strong homotopy $f',$
such that the sequence $h_n$ defines a homotopy between $f$ and
$f'.$ Again by Lemma \ref{perturb_f_n}, we have that
$f_i'=\phi^*(\widetilde{f_i})$ for $i\leq n.$ This proves
Lemma.\end{proof} Injectivity is proved.
\end{proof}

\subsection{Proof of Main Theorem}
\label{proof}

Before we prove Main Theorem, we need one more lemma.

\begin{lemma}\label{tautology}Let $\phi:\cC\to \cD$ be an equivalence of graded pre-categories. Let $m,m'$ be some
$A_{\infty}\text{-}$structures on $\cC$ and $\cD$ respectively.
The following are equivalent:

(i) The functor $\phi$ can be extended to an
$A_{\infty}\text{-}$functor $\Phi:(\cC,m)\to (\cD,m');$

(ii) The $A_{\infty}\text{-}$structures $m$ and
$\phi^*(m'),$ are strongly homotopic.
\end{lemma}

\begin{proof}Evident.\end{proof}

\begin{proof}[Proof of Theorem \ref{main_theo}.] By Lemmas
\ref{ess.small-small} and \ref{red_to_min}, it suffices to prove
that quasi-equivalence classes of small minimal
$A_{\infty}\text{-}$categories are in bijection with
quasi-equivalence classes of small minimal
$A_{\infty}\text{-}$pre-categories.

Given a minimal $A_{\infty}\text{-}$category $\cC,$ it can also be
considered as an $A_{\infty}\text{-}$pre-category with
$\cC_{tr}^n=Ob(\cC)^n.$ Clearly, if $\cC$ and $\cD$ are
quasi-equivalent minimal $A_{\infty}\text{-}$categories, then the
associated minimal $A_{\infty}\text{-}$pre-categories are
quasi-equivalent.

Now, let $\cC$ be a minimal $A_{\infty}\text{-}$pre-category.
Denote by $m$ the $A_{\infty}\text{-}$structure on $\cC^{gr}$
corresponding to $\cC.$

We have an equivalence of graded pre-categories
$\iota_{\cC^{gr}}:\cC^{gr}\to \cC^{gr}_{full}.$ By Theorem
\ref{invar_theo}, there exists an $A_{\infty}\text{-}$structure
$\widetilde{m}$ on $\cC^{gr}_{full},$ such that the
$A_{\infty}\text{-}$structure $\iota_{\cC^{gr}}^*(\widetilde{m})$
is strongly homotopic to $m.$ By lemma \ref{tautology}, the functor
$\iota_{\cC^{gr}}$ can be extended to the quasi-equivalence
$\cC\to (\cC^{gr}_{full},m).$ Hence, starting from a minimal
$A_{\infty}\text{-}$pre-category, we constructed some minimal
$A_{\infty}\text{-}$category $\widetilde{\cC},$ together with a
quasi-equivalence $\cC\to \widetilde{\cC}.$ We are left to prove
that, starting from quasi-equivalent
$A_{\infty}\text{-}$pre-categories, we obtain quasi-equivalent
$A_{\infty}\text{-}$categories.

Let $F:\cC\to\cD$ be a quasi-equivalence of minimal
$A_{\infty}\text{-}$pre-categories. The functor of graded
categories $F_1:\cC^{gr}\to \cD^{gr}$ can be obviously extended to
a functor $\Phi_1:\cC^{gr}_{full}\to \cD^{gr}_{full},$ so that we
have a commutative square of functors:

\begin{equation}\label{comm_funct}
\begin{CD}
\cC^{gr} @>F_1>> \cD^{gr}\\
@V\iota_{\cC^{gr}} VV                              @V\iota_{\cD^{gr}} VV\\
\cC^{gr}_{full} @>\Phi_1 >> \cD^{gr}_{full}.
\end{CD}
\end{equation}

Denote by $m$ (resp. $m'$) the $A_{\infty}\text{-}$structure on
$\cC^{gr}$ (resp. on $\cD^{gr}$) corresponding to $\cC$ (resp. to
$\cD$). Further, denote by $\widetilde{m}$ (resp.
$\widetilde{m'}$) the $A_{\infty}\text{-}$structure on
$\cC^{gr}_{full}$ (resp. on $\cD^{gr}_{full}$) such that
$\iota_{\cC^{gr}}^*(\widetilde{m})$ is strongly homotopic to $m$
(resp. $\iota_{\cD^{gr}}^*(\widetilde{m'})$ is strongly homotopic
to $m'$). By Lemma \ref{tautology}, $A_{\infty}\text{-}$structures
$F_1^*(m')$ and $m$ are also strongly homotopic. Hence, from the
commutative square \ref{comm_funct} and Theorem \ref{invar_theo},
we obtain that $A_{\infty}\text{-}$structures
$\Phi_1^*(\widetilde{m'})$ and $\widetilde{m}$ are strongly
homotopic. Therefore, by Lemma \ref{tautology}, the functor
$\Phi_1$ can be extended to a quasi-equivalence
\begin{equation}\Phi:(\cC^{gr}_{full},\widetilde{m})\to \cD^{gr}_{full},\widetilde{m'}).\end{equation}
Thus, starting from quasi-equivalent minimal
$A_{\infty}\text{-}$pre-categories, we obtain quasi-equivalent
minimal $A_{\infty}\text{-}$categories. Theorem is proved.
\end{proof}

\section{Twisted complexes over $A_{\infty}\text{-}$pre-categories}
\label{twisted}

It is clear that Main Theorem implies that we can take
pre-triangulated envelope and perfect derived category of any
essentially small $A_{\infty}\text{-}$pre-category, by replacing
it with some quasi-equivalent $A_{\infty}\text{-}$category.
However, it is useful in practice to have a construction of
pre-triangulated envelope in the framework of
$A_{\infty}\text{-}$pre-categories. We present such a construction
in this section. It is in fact straightforward generalization of
twisted complexes over ordinary $A_{\infty}\text{-}$categories
\cite{K}.

We work here over arbitrary graded commutative ring $k.$

Let $\cC$ be an $A_{\infty}\text{-}$pre-category. Define the
$A_{\infty}\text{-}$pre-category $\widetilde{\cC}$ as follows:

1) $Ob(\widetilde{\cC})=\{X[n]\mid X\in Ob(\cC), n\in\Z\};$

2)
$(X_1[n_1],\dots,X_k[n_k])\in\widetilde{\cC}_{tr}^k\Leftrightarrow
(X_1,\dots,X_k)\in \cC_{tr}^k.$

3) For a transversal pair $(X_1[n_1],X_2[n_2]),$ we put
\begin{equation}\Hom(X_1[n_1],X_2[n_2]):=\Hom(X_1,X_2)[n_2-n_1].\end{equation}
The higher products equal to that in $\cC.$

\smallskip

Now, we define the twisted complexes. Given a transversal sequence
$S=(X_1,\dots,X_n)\in \cC_{tr}^n,$ we put
\begin{equation}\End_+(S):=\bigoplus_{1\leq i<j\leq n}\Hom(X_i,X_j).\end{equation}

\begin{defi}A non-unital $A_{\infty}\text{-}$algebra $A$ is called nilpotent if it is equipped with finite
decreasing filtration $A=F_1A\supset F_2A\supset\dots F_nA=0,$
such that
\begin{equation}m_k(F_{r_1}A\otimes\dots\otimes F_{r_k}A)\subset F_{r_1+\dots +r_k}A.\end{equation}\end{defi}

Clearly, $\End_+(S)$ is a nilpotent $A_{\infty}\text{-}$algebra,
with filtration $F_{\cdot}\End_+(S),$ where
\begin{equation}F_r\End_+(S)=\bigoplus\limits_{1\leq i\leq j-r\leq n-r}\Hom(X_i,X_j),\quad r\geq 1.\end{equation}

Just as in \cite{ELO2}, for each nilpotent
$A_{\infty}\text{-}$algebra $A,$ we have a groupoid $\cM\cC(A)$ of
Maurer-Cartan solutions in $A.$ The following Lemma is straightforward
generalization of \cite{ELO2}, Theorem 7.2.

\begin{lemma}\label{bijection}Let $f:A_1\to A_2$ be a filtered $A_{\infty}\text{-}$quasi-isomorphism of nilpotent
$A_{\infty}\text{-}$algebras. Then the induced functor
\begin{equation}f^*:\cM\cC(A_1)\to \cM\cC(A_2)\end{equation} is an
equivalence.\end{lemma}

\begin{defi}\label{C^pre-tr}Let $\cC$ be an $A_{\infty}\text{-}$pre-category. Define the $A_{\infty}\text{-}$pre-category
$\cC^{pre-tr}$ of twisted complexes over $\cC$ as follows:

1) Objects of $\cC^{pre-tr}$ are pairs $(S,\alpha),$ where $S$ is
some transversal sequence in $\widetilde{\cC},$ and $\alpha\in
\End_+(S)^1$ is a Maurer-Cartan solution.

2) The sequence $((S_1,\alpha_1),\dots,(S_n,\alpha_n))\in
(\cC^{pre-tr})^n$ is transversal iff the sequence
$(S_1,\dots,S_n)$ is transversal in $\widetilde{\cC}.$

3) For a transversal pair $((S_1,\alpha_1),(S_2,\alpha_2))$ in
$\cC^{pre-tr},$ we put
\begin{equation}\Hom((S_1,\alpha_1),(S_2,\alpha_2))=\bigoplus\limits_{X\in S_1,Y\in S_2}\Hom(X,Y).\end{equation}

4) For a transversal sequence,
$((S_0,\alpha_0),\dots,(S_n,\alpha_n))\in
(\cC^{pre-tr})_{tr}^{n+1},$ and homogeneous morphisms $x_i\in
\Hom((S_{i-1},\alpha_{i-1}),(S_i,\alpha_i))$ we put
\begin{equation}m_n(x_n,\dots,x_1)=\sum\limits_{k_0,\dots,k_n\geq 0}(-1)^{\epsilon}
m_{n+k_0+\dots+k_n}(\alpha_n^{k_n},x_n,\alpha_{n-1}^{k_{n-1}},\dots,x_1,\alpha_0^{k_0}),\end{equation}
where $\epsilon=\sum\limits_{n\geq i>j\geq
0}(\deg(x_i)+k_i)k_j+\sum\limits_{i=0}^n
\frac{k_i(k_i+1)}2+\sum\limits_{i=1}^n ik_i.$\end{defi}

\begin{prop}The $A_{\infty}\text{-}$pre-category $\cC^{pre-tr}$ is
well-defined.
\end{prop}
\begin{proof} The only non-obvious thing to check is the
extension property. To prove it, we need the following lemma.

\begin{lemma}\label{good_q-is}Let $\cD$ be an $A_{\infty}\text{-}$pre-category, and
$(X_1,\dots,X_n,Y_1,\dots,Y_n)$ a transversal sequence in $\cD.$
Suppose that we are given with quasi-isomorphisms $F_i:X_i\to
Y_i,$ $1\leq i\leq n.$ Then for each Maurer-Cartan solution
$\beta\in \End_+(Y_1,\dots,Y_n)$ (resp. $\alpha\in
\End_+(X_1,\dots,X_n)$) there exists a Maurer-Cartan solution
$\alpha\in \End_+(X_1,\dots,X_n)$ (resp. $\beta\in
\End_+(Y_1,\dots,Y_n)$), together with a quasi-isomorphism
$G:((X_1,\dots,X_n),\alpha)\to ((Y_1,\dots,Y_n),\beta)$ in
$\cD^{pre-tr}.$
\end{lemma}

\begin{proof} Consider the following $A_{\infty}\text{-}$algebras:
$\cA_1=\End_+(X_1,\dots,X_n),$ $\cA_2=\End_+(Y_1,\dots,Y_n)$ and
$\cB$ described as follows. As a $k\text{-}$module,
\begin{equation}\cB=\cA_1\oplus\cA_2\oplus \bigoplus_{1\leq i<j\leq n}\Hom(X_i,Y_j)[-1],\end{equation}

and the higher products are direct sums of that in $\cD,$ and also
\begin{multline}m_{k+l+1}(y_l,\dots,y_1,F_p,x_k,\dots,x_1)\quad\text{ for }y_i\in
\Hom(Y_{j_{i-1}},Y_{j_i}), x_i\in \Hom(X_{q_{i-1}},X_{q_i}),\\
1\leq q_0<\dots<q_k=p=j_0<\dots<j_l\leq n.
\end{multline}

We have obvious projections $\pi_i:\cB\to \cA_i,$ $i=1,2,$ which
are quasi-isomorphisms.

Now suppose that $\beta\in \cA_2^1$ is an MC solution. We will
show how to construct the required $\alpha\in \cA_1^1.$ The
construction in the other direction is analogous.

By Lemma \ref{bijection}, there exists an MC solution
$\tilde{\beta}\in \cB^1$ such that $\pi_2(\tilde{\beta})$ is
homotopic to $\beta.$ The components of $\tilde{\beta},$ together
with $F_i:X_i\to Y_i$ give objects
$E_1=((X_1,\dots,X_n),\pi_1(\tilde{\beta})),
E_2=((Y_1,\dots,Y_n),\pi_2(\tilde{\beta}))\in Ob(\cD^{pre-tr}),$
and a quasi-isomorphism $F:E_1\to E_2$ (this is straightforward).

Further, if $h\in \cA_2^0$ is a morphism
$\pi_2(\tilde{\beta})\to\beta$ in the groupoid $\cM\cC(\cA_2),$
then $G:E_1\to ((Y_1,\dots,Y_n),\beta),$
\begin{equation}G=F+\sum\limits_{k_0,k_1,k_2\geq 0}m_{2+k_0+k_2}^{\cD}(\beta^{k_2},h,\pi_2(\tilde{\beta})^{k_1},F,
\pi_1(\tilde{\beta})^{k_0})\end{equation} is a quasi-isomorphism
in $\cD^{pre-tr}.$ This proves Lemma.
\end{proof}

The above Lemma immediately implies the extension property for
$\cD^{pre-tr}.$
\end{proof}

\begin{prop}\label{invariance} 1) Each $A_{\infty}\text{-}$functor $F:\cD_1\to \cD_2$ induces
in the natural way an $A_{\infty}\text{-}$functor
$F^*:\cD_1^{pre-tr}\to \cD_2^{pre-tr}.$

2) In the case when $F$ is a quasi-equivalence, $F^*$ is also
such.
\end{prop}
\begin{proof}1) The $A_{\infty}\text{-}$functor $F^*$ is given by the same formulas as in \cite{ELO2}, Section 7. We should
prove that it preserves quasi-isomorphisms. We note that this is
evident for the following class of "good"
quasi-isomorphisms.

We call a quasi-isomorphism $f:((X_1,\dots,X_n),\alpha)\to
((Y_1,\dots,Y_m),\beta)$ good if $n=m,$ the components
$f_{ij}:X_i\to Y_j$ vanish for $i>j$ and the components
$f_{ii}:X_i\to Y_i$ are quasi-isomorphisms.

Clearly, good quasi-isomorphisms are preserved by $F^*.$ It
follows from Lemma \ref{good_q-is} that for each quasi-isomorphism
$f_1:E_1\to E_2$ in $\cD_1^{pre-tr},$ there exist
quasi-isomorphisms $f_0:E_0\to E_1,$ $f_2:E_2\to E_3$ such that
the sequence $(E_0,E_1,E_2,E_3)$ is transversal and
$m_2(f_1,f_0),$ $m_2(f_2,f_1)$ are homotopic to good quasi-isomorphisms. This
proves part 1) of Proposition.

2) If $F$ is a quasi-isomorphism, then all morphisms
$f_1:\Hom_{\cD_1^{pre-tr}}(E_1,E_2)\to
\Hom_{\cD_2}(F(E_1),F(E_2))$ induce quasi-isomorphisms on the
subquotients with respect to the natural filtrations

\begin{equation}F_r\Hom(((X_1,\dots,X_n),\alpha),((Y_1,\dots,Y_m),\beta))=\bigoplus\limits_{j-i\geq r} \Hom(X_i,Y_j).\end{equation}

Essential surjectivity is implied by Lemma \ref{good_q-is},
together with Lemma \ref{bijection}.
\end{proof}

For the ordinary $A_{\infty}\text{-}$categories, our construction
gives standard $A_{\infty}\text{-}$categories of twisted complexes
introduced in \cite{BK} for DG categories and generalized in
\cite{K} to $A_{\infty}\text{-}$categories.

Now suppose that $k$ is again a field. By Proposition
\ref{invariance} 2), we have that passing from an essentially
small $A_{\infty}\text{-}$pre-category to quasi-equivalent
$A_{\infty}\text{-}$category commutes (up to quasi-equivalence)
with taking of twisted complexes.

\end{document}